\documentclass[a4paper,reqno]{amsart}
\usepackage{amssymb, amsmath, amscd}
\usepackage[final]{graphicx}
\usepackage{float}\usepackage{wrapfig}
\usepackage{color}
\usepackage{bbm}
\usepackage[final]{graphicx}

\newtheorem{theorem}{Theorem}[section]

\newtheorem{lemma}[theorem]{Lemma}

\makeatletter
 \@addtoreset{equation}{section}
\makeatother
\begin{document}
\title{An almost Zoll affine surface}
\author{Peter B Gilkey}
\address{Mathematics Department, University of Oregon, Eugene OR 97403-1222, US} \email{gilkey@uoregon.edu}
\begin{abstract}{An affine surface is said to be an affine Zoll surface if all affine geodesics
close smoothly. It is said to be an affine almost Zoll surface if thru any point, every affine geodesic
but one closes smoothly (the exceptional geodesic is said to be alienated as it does not return). 
We exhibit an affine structure on the cylinder which is almost Zoll. This
structure is geodesically complete, affine Killing complete, and affine symmetric.\newline Key words: affine manifold, Zoll surface
\newline Subject class 53C21}
\end{abstract}
\maketitle
\section{Introduction} A Riemannian manifold is said to be a {\it Zoll manifold} the geodesics
are all simple closed curves of the same length. Zoll~\cite{Z03} showed that the sphere $S^2$ admits
many such metrics in addition to the round one. Later authors called these surfaces {\it auf Wiedersehensfl\"achen}
or ``until we meet again" surfaces. In Spanish, this becomes {\it te veo de nuevo superficie}.

The only compact surfaces which admit Zoll metrics
are the sphere $S^2$ and real projective space $\mathbb{RP}^2$.
 Green~\cite{G63} showed that the only Zoll metric on $\mathbb{RP}^2$
is the standard homogeneous metric, up to isometry and rescaling; this was later extended by
Pries~\cite{P08} to show that if all the geodesics of a metric on $\mathbb{RP}^2$ are closed, then
the metric is the standard homogeneous metric, up to isometry and rescaling. One can consider analogous
questions in the Lorentzian setting -- see, for example, the work of Mounoud and Suhr~\cite{MS16}. Zoll surfaces
have been used in many contexts; see, for example, the work on Balacheff et al.~\cite{Bx09} concerning a geodesic
length conjecture. Zoll surfaces also have been considered in the orbifold context; see, for example, Lange~\cite{L17}.
We refer to Besse~\cite{B78} for a discussion
of more of the history of the subject than is available here and also to the references cited above.

An affine surface $\mathcal{M}$ is a pair $(M,\nabla)$ where $M$ a 2-dimensional manifold $M$ and $\nabla$ is a torsion
free connection on the tangent bundle of $M$. A diffeomorphism $\Phi$ from $(M_1,\nabla_1)$ to $(M_2,\nabla_2)$ is said to be {\it affine}
if $\Phi$ intertwines $\nabla_1$ and $\nabla_2$. An affine surface $\mathcal{M}$ is said to be {\it homogeneous} if the group of affine
diffeomorphisms acts transitively on $\mathcal{M}$. A vector field $X$ on $\mathcal{M}$ is said to be an {\it affine Killing vector field}
if the (locally defined) flows of $X$ are (locally defined) affine diffeomorphisms of $\mathcal{M}$ or, equivalently by Kobayashi and Nomizu~\cite{KN63}
the Lie derivative $\mathcal{L}_X\nabla=0$. The Lie bracket makes the set $\mathcal{K}(\mathcal{M})$ of affine Killing vector fields into a Lie algebra.

Let $R(X,Y):=\nabla_X\nabla_Y-\nabla_Y\nabla_X-\nabla_{[X,Y]}$ be the curvature
operator. The Ricci tensor is given by $\rho(X,Y):=\operatorname{Tr}(Z\rightarrow R(X,Y)Z)$. Because the Ricci tensor $\rho$
need not be symmetric in the affine context, one introduces the {\it symmetric Ricci tensor}
 $\rho_s(X,Y):=\frac12\{\rho(X,Y)+\rho(Y,x)\}$.  We say two affine structures $\nabla$ and $\tilde\nabla$ are {\it projectively equivalent}
if there exists a smooth 1-form $\omega$ so that one may express $\tilde\nabla_XY=\nabla_XY+\omega(X)Y+\omega(Y)X$.
Two affine structures are projectively equivalent if and only if their unparametrized geodesics coincide; see for example,
Schouten~\cite{S54};
thus it is natural to pass to projective structures. Note that projective equivalence does not preserve geodesic completeness since the
parametrized geodesics will in general be different. The geometries are said to be {\it strongly projectively equivalent} if $\omega$ is exact, i.e.
$$
\tilde\nabla_XY=\nabla_XY+X(\phi)Y+Y(\phi)X\quad\text{for some}\quad\phi\in C^\infty(M)\,.
$$
In this setting, 
we shall say that $\phi$ provides a {\it strong projective equivalence} from $\mathcal{M}=(M,\nabla)$ to 
${}^\phi\mathcal{M}:=(M,\tilde\nabla)$. We say that $\mathcal{M}$ is {\it strongly
projectively flat} if $\mathcal{M}$ is strongly projectively equivalent to a flat connection.
There is a question of which projective structures can be metrized, i.e.
arise as the Levi-Civita connection of some metric; we refer to Bryant et al.~\cite{B09} for further details concerning this question.

One says that an affine surface is an {\it affine Zoll surface}
if all of the geodesics are simple closed curves; this is a projective question. LeBrun and Mason~\cite{LM02}
(see also later related work in \cite{LM09})
showed that the only compact surfaces which admit affine Zoll structures are $\mathbb{S}^2$ and $\mathbb{RP}^2$.
Consequently, it is natural to weaken the condition just a little. We say that an affine surface is {\it almost Zoll} if for every
point of the surface, there is a single geodesic, which will be called the {\it alienated geodesic}, which does not close
and such that all other geodesics thru that point are simple closed curves which return to the initial point.
In this brief note, we present several examples of this phenomena. In Section~\ref{S2}, we discuss the quasi-Einstein equation
and present its basic properties in Theorem~\ref{Thm1}.
In section~\ref{S3}, we introduce the affine surfaces $\mathcal{M}(c)$ that will form the focus of our investigation. In Theorem~\ref{Thm2},
we show that $\mathcal{M}(c)$ is affine homogeneous, we determine the Lie algebra of affine Killing vector fields of $\mathcal{M}(c)$,
and we determine the connected component $\mathcal{G}(c)$ of the 4-dimensional Lie group of affine diffeomorphisms. In
Theorem~\ref{Thm3}, we show the geometries $\mathcal{M}(c)$ and $\mathcal{M}(\tilde c)$ are neither strongly projectively equivalent
nor locally affine equivalent for $c\ne\tilde c$.

 \section{The quasi-Einstein equation}\label{S2}
 The solutions to the quasi-Einstein equation is a fundamental invariant in the theory of affine surfaces.
We define the {\it Hessian} by setting:
$$
\mathcal{H}\phi=(\partial_{x^i}\partial_{x^j}\phi-\Gamma_{ij}{}^k\partial_{x^k}\phi)dx^i\otimes dx^j\in S^2(M)\,.
$$
Let
$\mathcal{Q}(\mathcal{M})$ be the solution space of the {\it quasi-Einstein equation}:
$$
\mathcal{Q}(\mathcal{M}):=\{\phi\in C^\infty(M):\mathcal{H}\phi+\phi\rho_s=0\}\,.
$$
There is a close relationship between strong projective equivalence and the solutions of the quasi-Einstein equation.
We refer to Brozos-V\'{a}zquez et al.~\cite{BGGV16} and to Gilkey and Valle-Regueiro~\cite{GV18} for the proof of the
following result as well as a further discussion of the quasi-Einstein equation and its use through the modified Riemannian extension
to study Walker metrics on the cotangent bundle of an affine surface.
\begin{theorem}\label{Thm1}
Let $\mathcal{M}=(M,\nabla)$ be a simply connected affine surface.
\begin{enumerate}
\item $\dim\{\mathcal{Q}(\mathcal{M})\}\le 3$. 
\item $\dim\{\mathcal{Q}(\mathcal{M})\}=3$ if and only if $\mathcal{M}$ is strongly projectively flat.
\item $\mathcal{Q}({}^\phi\mathcal{M})=e^\phi\mathcal{Q}(\mathcal{M})$.
\item Let $\mathcal{M}_1$ and $\mathcal{M}_2$ be two strongly projectively flat affine surfaces and 
let $\Phi$ be a diffeomorphism from $M_1$ to $M_2$.
If $\Phi^*\mathcal{Q}(\mathcal{M}_2)=\mathcal{Q}(\mathcal{M}_1)$,
then $\Phi$ is an affine morphism from $\mathcal{M}_1$ to $\mathcal{M}_2$.
\end{enumerate}\end{theorem}

\section{The affine structures $\mathcal{M}(c)$}\label{S3}
The following family of surfaces was introduced by 
D'Ascanio et al.~\cite{D17} in their study of geodesically complete homogeneous affine surfaces.
Let $\mathcal{M}(c):=(\mathbb{R}^2,\nabla)$ 
where the only (possibly) non-zero Christoffel symbols of $\nabla$ are
 $\Gamma_{22}{}^1=(1+c^2)x^1$ and $\Gamma_{22}{}^2=2c$. Replacing $x^2$ by $-x^2$ interchanges the roles
 of $\mathcal{M}(c)$ and $\mathcal{M}(-c)$ so we shall assume $c\ge0$ henceforth.
A direct computation shows
 \begin{equation}\label{E1.a}
 \rho=(1+c^2)dx^2\otimes dx^2\,.
 \end{equation}
 An affine connection is said to be {\it symmetric} if and only if $\nabla\rho=0$; one verifies that $\mathcal{M}(c)$ is affine symmetric
 if and only if $c=0$.

\begin{theorem}\label{Thm2}
Let $\mathcal{G}(c)$ be the set of all smooth maps from $\mathbb{R}^2$ to $\mathbb{R}^2$ of the form:
\smallbreak\centerline{$
T(\alpha,\beta,\gamma,\delta)(x^1,x^2):=(e^{\alpha}x^1+\beta e^{cx^2}\cos(x^2)+\gamma e^{cx^2}\sin(x^2),x^2+\delta)$.}
\smallbreak\noindent Then $\mathcal{G}(c)$ is a 4-dimensional Lie group which is diffeomorphic to $\mathbb{R}^4$ which acts transitively 
on $\mathbb{R}^2$ so $\mathcal{M}(c)$ is a homogeneous space. The Lie algebra of $\mathcal{G}(c)$ is
\smallbreak\centerline{
$\mathfrak{g}(c):=\operatorname{Span}\{x^1\partial_{x^1},e^{cx^2}\cos(x^2)\partial_{x^1},e^{cx^2}\sin(x^2)\partial_{x^1},\partial_{x^2}\}$.}
\smallbreak\noindent The vector fields of $\mathfrak{g}(c)$ are all complete and $\mathfrak{g}(c)=\mathcal{K}(\mathcal{M}(c))$
is the set of affine Killing vector fields of
$\mathcal{M}(c)$.
\end{theorem}

\begin{proof} We follow the discussion in Gilkey et al.~\cite{GPV19}. A direct computation shows that
$\{e^{cx^2}\cos(x^2),e^{cx^2}\sin(x^2),x^1\}\subset\mathcal{Q}(\mathcal{M}(c))$.
Consequently, by Theorem~\ref{Thm1}, $\mathcal{M}(c)$ is strongly projectively flat and
\begin{equation}\label{E1.b}
\mathcal{Q}(\mathcal{M}(c))=\operatorname{Span}\{e^{cx^2}\cos(x^2),e^{cx^2}\sin(x^2),x^1\}\,.
\end{equation}
It is then immediate that $T(\alpha,\beta,\gamma,\delta)^*\mathcal{Q}({\mathcal{M}}(c))=\mathcal{Q}(\mathcal{M}(c))$ so
$T(\alpha,\beta,\gamma,\delta)$ is a diffeomorphism of $\mathbb{R}^2$ preserving the affine structure.  We show that $\mathcal{G}(c)$
is a Lie group with Lie algebra $\mathfrak{g}(c)$ by computing:
\begin{eqnarray*}
&&T(\alpha,\beta,\gamma,\delta)\circ T(\tilde\alpha,\tilde\beta,\tilde\gamma,\tilde\delta)\\
&&\quad= T(e^{\alpha+\tilde\alpha},e^{\alpha} \tilde\beta+\beta e^{c \tilde\delta} \cos (\tilde\delta)+\gamma e^{c \tilde\delta} \sin (\tilde\delta),\\
&&\qquad e^{\alpha} \tilde\gamma-\beta e^{c \tilde\delta} \sin (\tilde\delta)+\gamma e^{c \tilde\delta} \cos (\tilde\delta),\delta+\tilde\delta)\\
&&T(\alpha,\beta,\gamma,\delta)^{-1}=T(-\alpha,-e^{-\alpha-c \delta} (\beta \cos (\delta)-\gamma \sin (\delta)),\\
&&\qquad\qquad\qquad\qquad\qquad -e^{-\alpha-c \delta} (\beta \sin (\delta)+\gamma \cos (\delta)),-\delta).\qed
\end{eqnarray*}

Following the notation of Opozda~\cite{Op04}, an affine structure on $\mathbb{R}^2$ is said to be Type~$\mathcal{A}$
if the Christoffel symbols are constant; work of Arias-Marco and Kowalski~\cite{AMK08} and of Vanzurova~\cite{V13}
show that if the Ricci tensor of such a structure is non-zero, then it is not metrizable, i.e. it does not arise as the
Levi-Civita connection of a pseudo-Riemannian metric.
Let $\mathcal{M}_0(c)$ be the affine structure on $\mathbb{R}^2$ where the only (possibly) non-zero Christoffel symbols
are
$\Gamma_{11}{}^1=1$, $\Gamma_{22}{}^1=1+c^2$, and $\Gamma_{22}{}^2=2c$.
This is a Type~$\mathcal{A}$ geometric structure on $\mathbb{R}^2$
in the notation of Opozda~\cite{Op04}; it is linearly equivalent to the surfaces $\mathcal{M}_5^c$ of \cite{BGG18} but is in a
slightly more convenient form for our purposes.
 Let $(u^1,u^2)$ be coordinates on $\mathbb{R}^2$.  A direct computation shows
$$
\mathcal{Q}(\mathcal{M}_0(c))=\operatorname{Span}\{e^{cu^2}\cos(u^2),e^{cu^2}\sin(u^2),e^{u^1}\}\,.
$$
Let $\Phi(u^1,u^2)=(e^{u^1},u^2)$. We then have $\Phi^*\mathcal{Q}(\mathcal{M}(c))=\mathcal{M}_0(c)$ so
$\Phi$ is an affine embedding of $\mathcal{M}_0(c)$ in $\mathcal{M}(c)$. We have $\rho=(1+c^2)dx^2\otimes dx^2$.
Results of Brozos-V\'{a}zquez et al.~\cite{BGG18} show that if $\mathcal{M}$ is a Type~$\mathcal{A}$ structure on
$\mathbb{R}^2$ and if $\operatorname{Rank}\{\rho\}=1$, then $\dim\{\mathfrak{K}(\mathcal{M}(c))\}=4$.
Consequently, the Lie algebra of Killing vector fields $\mathcal{K}(\mathcal{M})$ is 4 -dimensional and is the Lie algebra of the Lie group $\mathcal{M}(c)$.
Consequently, every affine Killing vector field on $\mathcal{M}(c)$ is complete.
\end{proof}

Patera et al.~\cite{PSWZ76} classified the 4-dimensional Lie algebras. We follow their notation.  Let
$A_{4,12}$ be the 4-dimensional Lie algebra with bracket relations 
$[e_1,e_3]=e_1$, $[e_2,e_3]=e_2$, $[e_1,e_4]=-e_2$, $[e_2,e_4]=e_1$. Results of \cite{BGG18} show $\mathfrak{K}(\mathcal{M}_0(c))$
is isomorphic to $A_{4,12}$ and, consequently, $\mathfrak{K}(\mathcal{M}(c))$ is isomorphic to $A_{4,12}$ as well.

\begin{theorem}\label{Thm3}
\ \begin{enumerate}
\item The geometries $\mathcal{M}(c)$ and $\mathcal{M}(\tilde c)$ are not locally affine equivalent for $c\ne\tilde c$.
\item The geometries $\mathcal{M}(c)$ and $\mathcal{M}(\tilde c)$ are not strongly projectively equivalent for $c\ne\tilde c$.
\item Let $\phi(x^1,x^2)=(e^{cx^2}x^1,x^2)$ define a diffeomorphism of $\mathbb{R}^2$. Then 
$\phi^*\mathcal{M}(c)$ is strongly projectively equivalent to $\mathcal{M}(0)$.
\end{enumerate}\end{theorem}

\begin{proof}
Let $\alpha(c):=\nabla\rho(\partial_{x^2},\partial_{x^2};\partial_{x^2})^2\rho(\partial_{x^2},\partial_{x^2})^3$. Results of \cite{BGG18}
show that $\alpha(c)$ is an affine invariant in this setting. One computes $\alpha(\mathcal{M}(c))=\frac{16c^2}{1+c^2}$. This shows that
$\mathcal{M}(c)$ is locally affine equivalent to $\mathcal{M}(\tilde c)$ if and only if $c=\tilde c$; these geometries are all distinct.
Suppose $c\ne\tilde c$. There is no function $h$ so that $\mathcal{Q}(\mathcal{M}(c))=e^h\mathcal{Q}(\mathcal{M}(\tilde c))$.
Consequently, by Theorem~\ref{Thm1}, these geometries are not strongly projectively equivalent. However, we verify that 
$\phi^*\mathcal{Q}(\mathcal{M}(c))=e^{cx^2}\mathcal{Q}(\mathcal{M}(0))$. Consequently, by
Theorem~\ref{Thm1} $\phi^*\mathcal{M}(c)$ is strongly projectively equivalent to $\mathcal{M}(0)$. ~\qed
\end{proof}

The geometry $\mathcal{M}(0)$ is geodesically complete; the geometries $\mathcal{M}(c)$ for $c>0$ are geodesically incomplete.
The following result follows from a direct computation.

\begin{theorem}\label{Thm4} 
\ \begin{enumerate}
\item The curves $
\sigma_{u,v;a,b}(t)=\left\{\begin{array}[l]{ll}
(u\cos(bt)+\frac ab\sin(bt),bt+v)&\text{ if }b\ne0\\
(at+u,0)&\text{ if }b=0
\end{array}\right\}$ are geodesics for the geometry $\mathcal{M}(0)$
with initial position $(u,v)$ and initial velocity $(a,b)$. $\mathcal{M}(0)$ is geodesically complete.
\item For $c>0$, let $\tau(t)=\log(1+2bct)$. Then the curves
$$
\sigma_{u,v;a,b}(t)=\left\{\begin{array}{ll}
(\frac{\sqrt{1+2bct}}b(bu\cos(\frac{\tau(t)}{2c})+(a-bcu)\sin(\frac{\tau(t)}{2c}),v+\frac{\tau(t)}{2c})&\text{ if }b\ne0\\
(at+u,0)&\text{ if }b=0\end{array}\right\}
$$
are the geodesics of $\mathcal{M}(c)$ with initial position $(u,v)$ and initial velocity $(a,b)$. $\mathcal{M}(c)$ is geodesically incomplete since
these curves are only defined for the parameter range $1+2bct>0$. 
\end{enumerate}\end{theorem}

\begin{proof}One verifies directly the curves $\sigma_{u,v;a,b}$ are geodesics with the given initial conditions. Since they are defined
for all time if $c=0$, $\mathcal{M}(0)$ is geodesically complete. On the other hand, they fail to be defined when $1+2bct=0$ and thus
$|mathcal{M}(c)$ is geodesically incomplete for $c>0$.
\end{proof}

We give below a picture of the geodesic structure at the origin $(u,v)=(0,0)$; we emphasize, it makes no difference since the geometry
is homogeneous. The geodesics fall into 2 families; those with $b>0$ all meet at the point $(0,\pi)$; those with $b<0$ 
all meet at the point $(0,-\pi)$, and those with $b=0$ lie along the $x^1$ axis. We also present similar pictures for the geodesics
that start at $(u,v)=(1,0)$ and at $(u,v)=(2,0)$.
\begin{figure}[H]\caption{Geodesic structure}
\includegraphics[height=4cm,keepaspectratio=true]{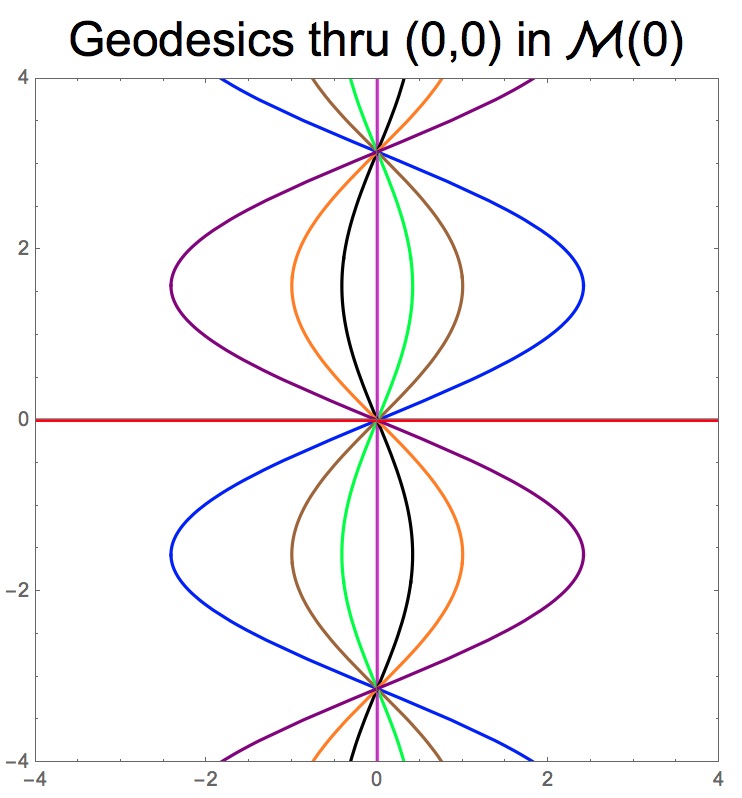}\quad
\includegraphics[height=4cm,keepaspectratio=true]{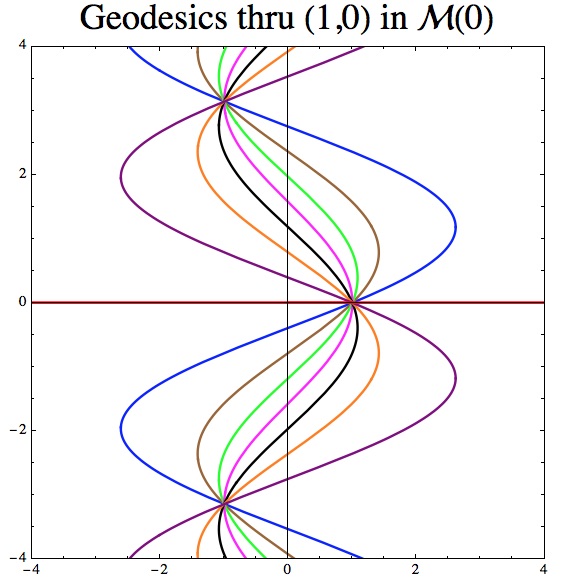}\quad
\includegraphics[height=4cm,keepaspectratio=true]{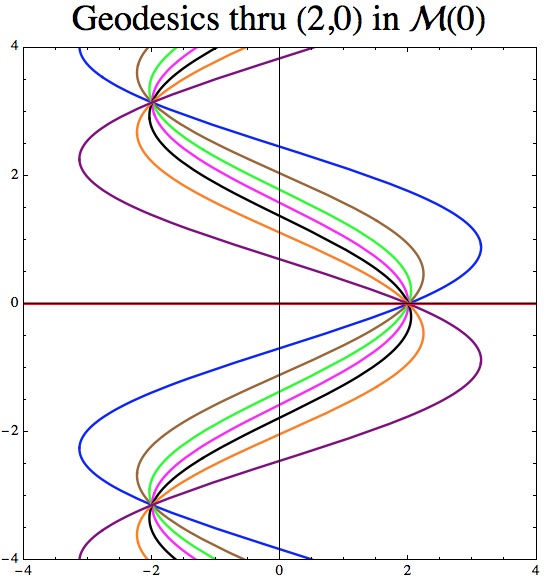}
\end{figure}

Let $\phi(x^1,x^2)=(e^{cx^2}x^1,x^2)$. By Theorem~\ref{Thm1}, $\phi^*\mathcal{M}(c)$ is strongly projectively equivalent to $\mathcal{M}(0)$.
Thus modulo the diffeomorphism $\phi$, the picture of the geodesic structure for $\mathcal{M}(c)$ is the same as that of $\mathcal{M}(0)$.

\subsection{An affine geometry on the cylinder}\label{S1.7} Let $\Phi(x^1,x^2)=(x^1,x^2+2\pi)$ generate
a fixed point free action of $\mathbb{Z}$
on $\mathbb{R}^2$; this corresponds to taking $\alpha=0$, $\beta=0$, $\gamma=0$, and
$\delta=\pi$ in Theorem~\ref{Thm2}.
We divide by this action to define an affine structure on the cylinder $\mathcal{C}:=(\mathbb{R}\times S^1,\nabla)$.
We verify immediately that if $c=0$, then
\begin{eqnarray*}
T(\alpha,\beta,\gamma,\delta)\Phi_k(x^1,x^2)&=&(e^{\alpha}x^1+\beta\cos(x^2+2k\pi)+\gamma\sin(x^2+2k\pi),x^2+2k\pi)\\
&=&(e^{\alpha}x^1+\beta\cos(x^2)+\gamma\sin(x^2),x^2+2k\pi)\\&=&\Phi_kT(\alpha,\beta,\gamma,\delta)(x^1,x^2)
\end{eqnarray*}
so $\Phi$ is in the center of the group $\mathcal{G}(0)$ and hence $\mathcal{G}(0)$ extends to a transitive
affine action on $\mathcal{C}$; this is a homogeneous geometry.
All the geodesics with a non-trivial vertical component close smoothly and where the
horizontal geodesic is the alienated geodesic. This is the desired quasi Zoll affine geometry. 
The exponential map is surjective on $\mathbb{R}^2$; it is not surjective on $\mathbb{R}^2$.
It is globally affine homogeneous, affine geodesically complete,
and affine Killing complete.
\subsection{An affine geometry on the M\"obius strip} Let $\Psi(x^1,x^2):=(-x^1,x^2+\pi)$; this
generates a fixed point free action of $\mathbb{Z}$ on $\mathbb{R}^2$. Let
$\mathcal{L}:=\mathbb{R}^2/\Psi(\mathbb{Z})$ be the quotient; this is the M\"obius strip.
In a purely formal sense,
this corresponds to taking
$\alpha=\pi\sqrt{-1}$, $\beta=\gamma=0$, and $\delta=\pi$ in Theorem~\ref{Thm2}. We compute
\begin{eqnarray*}
T(\alpha,\beta,\gamma,\delta)\Psi(x^1,x^2)&=&(-e^\alpha x^1+\beta\cos(x^2+\pi)+\gamma\sin(x^2+\pi),x^2+\pi)\\
&=&(-e^\alpha x^1-\beta\cos(x^2)-\gamma\sin(x^2),x^2+\pi)\\
&=&\Psi T(\alpha,\beta,\gamma,\delta)(x^1,x^2)
\end{eqnarray*}
Thus this is a homogeneous affine structure as well and we have a double cover $\mathbb{Z}_2\rightarrow\mathcal{C}\rightarrow\mathcal{L}$ on which $\mathcal{G}(0)$ acts equivariantly. 
With this structure, the M\"obius strip is affine homogeneous, geodesically complete, affine complete, and almost Zoll.

Let $\tau(t)=\log(1+2bct)$. Then
$$
\sigma_{u,v;a,b}(t)=\left\{\begin{array}{ll}
(\frac{\sqrt{1+2bct}}b(bu\cos(\frac{\tau(t)}{2c})+(a-bcu)\sin(\frac{\tau(t)}{2c}),v+\frac{\tau(t)}{2c})&\text{ if }b\ne0\\
(at+u,0)&\text{ if }b=0\end{array}\right\}\,.
$$
This geometry is geodesically incomplete; it is only defined for the parameter range $1+2bct>0$. 
Still, all geodesics thru the origin either focus vertically above or below the $x^1$-axis
or are horizontal and the general pattern is the same. A geodesic $\sigma$ is an alienated geodesic 
if and only $\rho(\dot\sigma,\dot\sigma)=0$. Dividing by $\mathbb{Z}$ yields
an affine quasi Zoll geometry. This geometry is locally homogeneous but not globally homogeneous since 
$\Phi$ is not in the center
of $\mathcal{G}(c)$ for $c\ne0$ owing to the presence of the exponential factor $e^{cx^2}$.

Let $\check T(\alpha,\delta)(x^1,x^2)=(\alpha x^1,x^2+\delta)$ define an affine action of 
$(\mathbb{R}-\{0\})\times\mathbb{R})$ on $\mathbb{R}^2$. This action commutes with 
$\Phi$; there are 2 $\mathcal{C}$ orbits -- the horizontal axis
and everything else. Thus this geometry is ``almost" affine homogeneous; the complement of the alienated geodesic
thru (0,0) is homogeneous as is the alienated geodesic thru (0,0). We also obtain an almost Zoll geometry on the
M\"obius strip.

\subsection{Speed} We have $\rho=(1+c^2)dx^2\otimes dx^2$. We use $\rho$ to define a positive semi-definite inner
product and let $\rho(X,X)$ be the ``speed". We suppose $c\ne0$. We compute
$$\rho(\dot\sigma_{u,v,a,b},\dot\sigma_{u,v,a,b})=\left\{\begin{array}{ll}
(1+c^2)\frac{2b}{(1+2bct)^2}&\text{if }b\ne0\\0&\text{if }b=0\end{array}\right\}\,.
$$
Thus the alienated geodesics are the null geodesics in the geometry $\mathcal{M}(c)$. Although the remaining
geodesics all return to the basepoint in the cylinder, the speed increases to $\infty$ as $t\rightarrow-\frac1{2bc}$;
the return to the basepoint occurs more and more rapidly.

\subsection{The projective tangent bundle} We digress briefly to relate this example
to the results of LeBrun and Mason~\cite{LM02}. If $M$ is a smooth manifold, let $\mathbb{P}(M)$ be the projective
tangent bundle. If $\nabla$ is an affine structure on $M$, the tangent of lifted geodesics defines a natural
foliation on $\mathbb{P}(M)$; the affine structure is said to be {\it tame} if this structure
on $\mathbb{P}(M)$ is locally trivial. LeBrun and Mason showed (see Lemma~2.7 and Lemma~2.8)
\begin{lemma} Suppose that $\mathcal{M}$ is an affine tame Zoll manifold. 
\begin{enumerate}
\item The universal cover of $M$ with the induced affine structure is tame Zoll.
\item $M$ is compact and any two points of $M$ can be joined by an affine geodesic.
\item $M$ has finite fundamental group.
\end{enumerate}
\end{lemma}

Our examples show that these results fail for almost Zoll structures. Let $\mathcal{M}$ is be an almost Zoll surface.
Associating to any point of $M$ the tangent to the alienated geodesic through that point defines a natural section
to $\mathbb{P}(M)$; let $\tilde{\mathbb{P}}(M)$ be the complement of this section. We adopt the notation of
Section~\ref{S1.7} to define $\mathcal{C}$; the alienated geodesics are the horizontal geodesics;
$\sigma_{u,v;a,b}$ is not alienated if and only if $b\ne0$. Since we are working projectively, we may set $b=1$. Let 
$\chi(r,s,t):=\dot\sigma_{r,0,s,1}(t)$ parametrize $\tilde{\mathbb{P}}(\mathcal{C})$; this identifies
$\tilde{\mathbb{P}}(\mathcal{C})$ with $\mathbb{R}\times\mathbb{R}\times S^1$ since, of course, $\sigma_{r,0,s,1}(t)$
is periodic with period $2\pi$ in $t$. This shows the foliation of $\tilde{\mathbb{P}}(TM)$ by lifted geodesics is a trivial
circle bundle over $\mathbb{R}^2$ and hence is tame. Similarly, if we lift to the universal cover, the 
foliation of $\mathbb{P}(T\mathbb{R}^2)-S(\mathbb{R}^2)$ by lifted geodesics is a trivial $\mathbb{R}$ bundle over 
$\mathbb{R}^2$ and hence tame. Clearly, however, the affine structure on $\mathbb{R}^2$ is no longer almost
Zoll and we can not join any two points of $\mathbb{R}^2$ by geodesics. And the cylinder does not have finite
fundamental group. On the other hand, the cylinder is the oriented double cover of the M\"obius strip.
\subsection{Global topology} As noted above, the tangent to the alienated geodesic thru any point of an almost Zoll manifold
is a section to $\mathbb{P}(TM)$. Consequently, if $M$ is compact, then the Euler-Poincare characteristic of $M$ vanishes.
Thus, in particular, the only compact surfaces which could potentially admit an almost Zoll structure are the torus and
the Klein bottle. The example we have constructed passes to the cylinder and the M\"obius strip; it does not pass to the
torus or the Klein bottle. We do not know if these admit an almost Zoll structure but we suspect the answer is no.

\section{Effect of the fundamental group} Up to affine equivalence, there is a unique surface with a $\mathbb{Z}_3$
symmetry. Let $T\in\operatorname{GL}(2,\mathbb{R})$ satisfy $T^3=\operatorname{id}$ and $T\ne\operatorname{id}$.
Let $e_2:=Te_1$ and set $Te_2=ae_1+be_2$ so that
$$T=\left(\begin{array}{ll}0&a\\1&b\end{array}\right)\quad\text{so}\quad
T^3=\left(\begin{array}{cc}a b & a^2+b^2 a \\b^2+a & a b+\left(b^2+a\right) b \\\end{array}\right)\,.
$$
The equation $T^3=\operatorname{id}$ forces $a=b=-1$ and we have
$$\begin{array}{ll}
T=\left(\begin{array}{cc}0&-1\\1&-1\end{array}\right),\quad&
T\left(\begin{array}{c}1\\0\end{array}\right)=\left(\begin{array}{c}0\\1\end{array}\right),\\
T\left(\begin{array}{c}0\\1\end{array}\right)=\left(\begin{array}{c}-1\\-1\end{array}\right),\quad&
T\left(\begin{array}{c}-1\\-1\end{array}\right)=\left(\begin{array}{c}1\\0\end{array}\right)\,.
\end{array}$$
We suppose $\mathcal{Q}=\operatorname{Span}\{e^{x^1},e^{x^2},e^{-x^1-x^2}\}$. Then $T^*\mathcal{Q}=\mathcal{Q}$.
Our ansatz tells us this is realized by
\begin{eqnarray*}
&&\Gamma\left(\frac{b^2+a-1,\ \ b^2-b,\ \ ba,\ \ ba,\ \ a^2-a,\ \ b+a^2-1}{b+a-1}\right)\\
&=&\Gamma\left(\frac{1}{3},-\frac{2}{3},-\frac{1}{3},-\frac{1}{3},-\frac{2}{3},\frac{1}{3}\right)\,.
\end{eqnarray*}
We let $M:=\{\mathbb{R}^2-\{0\}\}/\mathbb{Z}_3$; topologically, this is the cylinder. We give this the inherited affine
structure to define $\mathcal{M}$. We then have $\mathfrak{K}(\mathcal{M})=\{0\}$ and $\mathcal{Q}(\mathcal{M})$
is 1-dimensional and may be identified with $e^{x^1}+e^{x^2}+e^{-x^1-x^2}$. This is the type $\mathcal{A}$ surface with
$\operatorname{Rank}\{\rho\}=2$ with a $\mathbb{Z}_3$ symmetry. The Ricci
tensor takes the form
$$
\rho=\frac13\left(\begin{array}{cc}
-2&-1\\-1&-2
\end{array}\right)\,.
$$
This is the cusp point in the moduli space of negative definite Ricci tensors.
\subsection*{Research partially supported by Project MTM2016-75897-P (AEI/FEDER, Spain)}

\end{document}